\documentclass[a4paper,12pt,reqno]{amsart}

\usepackage[english]{babel}
\usepackage[T1]{fontenc} 
\usepackage[utf8]{inputenc}

\usepackage{listings}
\lstdefinelanguage{code}{
basicstyle=\small\ttfamily,
alsoletter=",
classoffset=1,
keywords={gb, eliminate, saturate, diff, degree, flatten, apply, tensor, product},
keywordstyle={\color{teal}},
classoffset=2,
morekeywords={from, to, list, terms, toList, entries, for, end, if, return},
keywordstyle={\color{blue}},
classoffset=3,
morekeywords={QQ},
keywordstyle={\color{teal}},
classoffset=4,
morekeywords={ideal, matrix, gens},
keywordstyle={\color{teal}},
xleftmargin=1.5cm,
xrightmargin=1em,
columns=fullflexible,
keepspaces=true,
stepnumber=1,
numbers=none,
captionpos=b,
showspaces=false,
frame=none
}

\usepackage{amsfonts} 
\usepackage{amsmath} 
\usepackage{amsthm} 
\usepackage{lmodern}
\usepackage{tikz-cd}

\usepackage[colorlinks]{hyperref}
\hypersetup{
  linkcolor=[RGB]{192,0,0},
  citecolor=[RGB]{0, 120, 0},
  urlcolor=[rgb]{0.6, 0.2, 0.2}
}
\usepackage[alphabetic]{amsrefs} 

\usepackage{amssymb} 

\usepackage[nameinlink]{cleveref}

\usepackage[dotinlabels]{titletoc}

\usepackage{xcolor} 

\usepackage{fix-cm}

\usepackage{bm}

\usepackage{derivative}

\usepackage{fancyhdr}

\usepackage{array}
\usepackage{braket}
\usepackage{scalefnt}
\usepackage{lipsum}

\usepackage{comment}
\usepackage{mathtools}
\usepackage{thmtools}
\usepackage{graphicx}
\usepackage[normalem]{ulem} 
\usepackage{tabularx}

\usepackage{paralist}
\usepackage{tensor}

\usepackage[font=small,labelfont=bf,margin=1.5cm]{caption}
\usepackage{algorithm}

\usepackage{tikz}
\usepackage{tikz-cd}
\usepackage{tikz-3dplot}
\usetikzlibrary{cd}
\usetikzlibrary{arrows}
\usetikzlibrary{matrix}
\usepackage{nicematrix} 
\usetikzlibrary{positioning}

\usepackage{xcolor}
\definecolor{red}{rgb}{1,0,0}
\definecolor{darkred}{RGB}{192,0,0}

\newcommand{\am}{\mathop{\rm am}\nolimits}
\newcommand{\lm}{\mathop{\rm lm}\nolimits}
\newcommand{\gm}{\mathop{\rm gm}\nolimits}
\newcommand{\sm}{\mathop{\rm sm}\nolimits}

\newcommand{\N}{\mathbb{N}}
\newcommand{\C}{\mathbb{C}}

\newcommand{\f}{\varphi}

\newcommand{\p}{\mathbb{P}}

\newcommand{\ot}{\otimes}

\newcommand{\bft}{\mathbf{t}}

\newcommand{\calS}{\mathcal{S}}
\newcommand{\calT}{\mathcal{T}}

\newcommand{\bbC}{\mathbb{C}}

\newcommand{\bbU}{\mathbb{U}}

\newcommand{\bbZ}{\mathbb{Z}}

\newcommand{\id}{\mathrm{id}}

\DeclareMathOperator{\res}{Res}
\DeclareMathOperator{\Res}{Res}

\DeclareMathOperator{\rk}{rk}

\newcommand{\GL}{\mathrm{GL}}

\usepackage[left=1in, right=1in, top=1in, bottom=1in, includefoot, headheight=13.6pt]{geometry}

\usepackage{enumitem}

\usepackage{adjustbox}

\allowdisplaybreaks

\numberwithin{equation}{section}
\theoremstyle{definition}
\newtheorem{defn}[equation]{Definition}
\newtheorem{definition}[equation]{Definition}
\theoremstyle{plain}
\newtheorem{teo}[defn]{Theorem}
\newtheorem{theorem}[defn]{Theorem}

\newtheorem{conjecture}[defn]{Conjecture}

\newtheorem{proposition}[defn]{Proposition}

\newtheorem{lemma}[defn]{Lemma}

\newtheorem{cor}[defn]{Corollary}
\newtheorem{corollary}[defn]{Corollary}
\theoremstyle{remark}

\newtheorem{remark}[defn]{Remark}

\newtheorem{example}[defn]{Example}
\allowdisplaybreaks

\counterwithin{figure}{section}

\usepackage[textwidth=0.8in]{todonotes}
\title{New conjectures on multiplicities of tensor eigenvalues}
\date{\today}

\author{Stefano Canino}
\address[Stefano Canino]{Dipartimento di Matematica, Università degli Studi di Trento, Via Sommarive 14, Povo, Trento, 38123, Italy. ORCID: \texttt{0000-0002-1606-4741}.} 
\email{stefano.canino@unitn.it}

\author{Cosimo Flavi}
\address[Cosimo Flavi]{Instytut Matematyczny Polskiej Akademii Nauk,
ul.~Śniadeckich 8, 00-656 Warsaw, Poland. ORCID: \texttt{0000-0001-5251-4474}.}
\email{cflavi@impan.pl}

\author{Francesco Galuppi}
\address[Francesco Galuppi]{Wydział Matematyki, Informatyki i Mechaniki, Uniwersytet Warszawski, Banacha 2, 02-097 Warsaw, Poland. ORCID: \texttt{0000-0001-5630-5389}.}
\email{galuppi@mimuw.edu.pl}

\author{Yuze Luan}
\address[Yuze Luan]{Instytut Matematyczny Polskiej Akademii Nauk,
ul.~Śniadeckich 8, 00-656 Warsaw, Poland. ORCID: \texttt{0009-0009-7111-0825}.}
\email{yluan@impan.pl}

\keywords{Tensor eigenvalues; eigenvectors; characteristic polynomial; algebraic multiplicity; geometric multiplicity; tensor rank.}

\subjclass{15A69, 15A42, 13P15, 14N07.}

\thanks{The authors are listed in alphabetical order. Canino has been funded by the Italian Ministry of University and Research in the framework of the Call for Proposals for scrolling of final rankings of the PRIN 2022 call -- Protocol no.~2022NBN7TL. Canino and Flavi have been funded by the project \textit{Thematic Research Programmes}, Action I.1.5 of the program \textit{Excellence Initiative -- Research University} (IDUB) of the Polish Ministry of Science and Higher Education. Galuppi acknowledges support by the National Science Center, Poland, under the project ``Tensor rank and its applications to signature tensors of paths'', 2023/51/D/ST1/02363.}

\begin{document}

\maketitle
\begin{abstract}
We work on two conjectures on tensor eigenvalue multiplicities. By using the language of algebraic geometry we give stronger, more refined versions of the conjectures and we prove them in many new cases, notably for all $2\times 2\times\dots\times 2$ tensors. We also establish a connection between the rank of a tensor and the multiplicities of the zero eigenvalue.
\end{abstract}

\section{Introduction}
Matrices are familiar from linear algebra, where they represent linear maps. They are widely studied via their characteristic polynomials, their eigenvalues and their eigenvectors. In mathematics, as well as in physics, engineering and many more applied sciences, higher degree maps are represented by tensors
, that is, 
elements of $\C^n\ot((\C^n)^*)^{\ot d}$
. A tensor
\begin{equation}
    \label{tensor decomposition}
T=\sum_{i=1}^r v_i\ot (\alpha_{i,1}\ot\dots\ot \alpha_{i,d})\in \C^n\ot((\C^n)^*)^{\ot d}
\end{equation} defines a linear map $(\C^n)^{\ot d}\to \C^n$ by
\[w_1\ot\dots\ot w_d\mapsto\sum_{i=1}^r \alpha_{i,1}(w_1)\cdots  \alpha_{i,d}(w_d)v_i.\]
Hence it induces an algebraic map $T:\C^n\to \C^n$ by $T(w)=T(w\ot\dots\ot w)$. If we choose a basis for $\C^n$, then the algebraic map $T$ is defined by $n$ homogeneous polynomials of degree $d$. Given how important eigenvectors and eigenvalues are for matrices, several generalizations have been proposed for tensors.
We are interested in the notion introduced in \cite{lim} and \cite{QI2005}. If $T,\bft\in\C^n\ot((\C^n)^*)^{\ot d}$, then a vector $w\in\C^n\setminus\{0\}$ is a \emph{$\bft$-eigenvector} of $T$ if there exists $\lambda_0\in\C$ such that $T(w)=\lambda_0\bft(w)$. In this case $\lambda_0$ is a \emph{$\bft$-eigenvalue} of $T$. The importance of tensor eigenvalues in mathematics and in science cannot be overestimated. Here we just point to the textbook reference \cite{QCC}.

In this paper we are concerned about multiplicities of tensor eigenvalues. While matrices have eigenspaces, the set $E_{\bft,T}(\lambda_0)$ of all $\bft$-eigenvectors of $T$ with $\bft$-eigenvalue $\lambda_0$ is not a linear space. However, it naturally comes with the structure of algebraic scheme - see \Cref{definition: multiplicity eigenscheme} - whose dimension $\gm_{\bft,T}(\lambda_0)$ is called the $\bft$-\emph{geometric multiplicity} of $\lambda_0$. On the other hand, the dimension of the linear space spanned by all the $\bft$-eigenvectors is called the $\bft$-\emph{span multiplicity} of $\lambda_0$ and denoted by $\sm_{\bft,T}(\lambda_0)$.

Just like matrices, tensors have a characteristic polynomial, as we recall in \Cref{section: characteristic polynomial}. The $\bft$-\emph{algebraic multiplicity} of $\lambda_0$ is its multiplicity $\am_{\bft,T}(\lambda_0)$ as a root of the $\bft$-characteristic polynomial of $T$. For matrices the span multiplicity is the same as the geometric multiplicity, and the latter cannot exceed the algebraic multiplicity. For tensors the relation among these notions of multiplicities is not fully understood yet.
 Here we recall \cite{QI2005}*{Conjecture 1} and \cite{multiplicities}*{Conjecture 1.1}. We take $T,\bft\in\C^n\ot((\C^n)^*)^{\ot d}$ and we make the technical assumption that $\bft$ is \emph{non-singular} - see \Cref{remark: an eigenvector has only one eigenvalue}.

\begin{conjecture}\label{conjecture: am vs sm} If $\lambda_0\in\C$
, then $\am_{\bft,T}(\lambda_0)\ge\sm_{\bft,T}(\lambda_0)$.
\end{conjecture} 

\begin{conjecture}\label{conjecture: algebraic vs geometric multiplicity}
Let $\lambda_0\in\C$
and let $E_1,\dots,E_s$ the irreducible components of $E_{\bft,T}(\lambda_0)$. Then
\begin{enumerate}[leftmargin=*]
\item\label{item: conjecture am gm strong} $\am_{\bft,T}(\lambda_0)\ge\sum_{i=1}^s \dim(E_i)d^{\dim(E_i)-1}$.
\item\label{item: conjecture am gm weak} $\am_{\bft,T}(\lambda_0)\ge \gm_{\bft,T}(\lambda_0)d^{\gm_{\bft,T}(\lambda_0)-1}$.
\end{enumerate}
\end{conjecture}
By definition there exists some $i\in\{1,\dots,s\}$ such that $\dim(E_i)=\gm_{\bft,T}(\lambda_0)$, therefore the second statement of \Cref{conjecture: algebraic vs geometric multiplicity} is a weaker version of the first one. In \cite{hypergraph} and \cite{ZHENG2024113780} \Cref{conjecture: algebraic vs geometric multiplicity} has been proven for some special families of tensors coming from hypergraphs, but we are far from a complete solution.

The purpose of this article is threefold. First we use the language of algebraic geometry to give stronger, more refined versions
of the two conjectures
, that take into account the scheme structure of $E_{\bft,T}(\lambda_0)$. In \Cref{section: resultants} we will illustrate how \Cref{conjecture: our version of am vs sm} and \Cref{conjecture: our version of am vs gm} imply \Cref{conjecture: am vs sm} and \Cref{conjecture: algebraic vs geometric multiplicity}, respectively.

\begin{conjecture}\label{conjecture: our version of am vs sm}
If $\lambda_0\in\C$, then $\am_{\bft,T}(\lambda_0)$ is at least the scheme-theoretic dimension of the linear span of $E_{\bft,T}(\lambda_0)$.
\end{conjecture}

\begin{conjecture}\label{conjecture: our version of am vs gm}
Let
$\lambda_0\in\C$
and call $E_1,\dots,E_s$ the irreducible components of $E_{\bft,T}(\lambda_0)$. Then
\[\am_{\bft,T}(\lambda_0)\ge\sum_{i=1}^s \deg(E_i)\dim(E_i)d^{\dim(E_i)-1}.\]
\end{conjecture}

Our second purpose is to provide new evidence for these conjectures. Here we collect some of our results from \Cref{section: geometric multiplicities n-1} and \Cref{section: multiplicities n=3}.
\begin{theorem}\label{theorem: summary contribution multiplicities}
Let
$\lambda_0\in\C$ be a $\bft$-eigenvalue of $T$.
\begin{enumerate}
\item If $\gm_{\bft,T}(\lambda_0)=n$, then  \Cref{conjecture: our version of am vs gm} holds.
\item If $\gm_{\bft,T}(\lambda_0)=n-1$, then  \Cref{conjecture: algebraic vs geometric multiplicity}\eqref{item: conjecture am gm weak} holds.
\item If every irreducible component of $E_{\bft,T}(\lambda_0)$ has dimension $n-1$, then \Cref{conjecture: our version of am vs gm} holds.
\item If $n=2$, then \Cref{conjecture: our version of am vs gm} holds.
\item If $n=3$, $\gm_{\bft,T}(\lambda_0)=1$ and the ideal from \Cref{definition: multiplicity eigenscheme} is saturated, then $\am_{\bft,T}(\lambda_0)\ge \frac{d^2}{2}+1$. In particular  \Cref{conjecture: algebraic vs geometric multiplicity}\eqref{item: conjecture am gm weak} holds.
\end{enumerate}
In all the cases above \Cref{conjecture: our version of am vs sm} holds as well.
\end{theorem}
Beside the statement in itself, we want to stress that our proof of  \Cref{theorem: summary contribution multiplicities} employs tools from algebraic geometry. We hope that these techniques, like specialization techniques or the Hilbert-Burch theorem, will be useful additions to the toolbox for the community interested in tensor spectral theory.

Our third contribution is to relate the multiplicity of the eigenvalues of $T$ to the \emph{rank} of $T$, that is the minimal $r$ such that a decomposition like \eqref{tensor decomposition} exists. The rank is arguably the most important invariant of a tensor: it generalizes matrix rank, it measures the complexity of $T$ and, as summarized in \cite{Lan}*{Section 1.3}, it has a wide range of applications. We study the connection between the rank and the multiplicities of the zero eigenvalues. As a byproduct, we show that our conjectures hold for small rank tensors.
\begin{theorem}\label{theorem: summary for low rank}
If $r=\rk(T)\in\{0,\dots,n\}$, then
\begin{equation}\label{equation: am gm low rank}
\gm_{\bft,T}(0)\ge n-r\mbox{ and }\am_{\bft,T}(0)\ge (n-r)d^{n-1},
\end{equation}
hence \Cref{conjecture: our version of am vs sm} holds.
Moreover, if $T$ is general among rank $r$ tensors then \eqref{equation: am gm low rank} are two equalities and so \Cref{conjecture: algebraic vs geometric multiplicity}\eqref{item: conjecture am gm weak} holds as well.
\end{theorem}

We choose to work over the field $\bbC$ 
because it is useful in order to apply some existing results. For example, in \Cref{section: resultants} we make use of the computation of the degree of the characteristic polynomial from \cite{FO14}
, which may fail if the field has positive characteristic
. In \Cref{remark: eigenscheme codim 1 principal ideal} we apply the characterization of schemes defined by principal ideals \cite{Hart}*{Proposition I.1.13} and the Chevalley theorem \cite{Hart}*{Exercise II.3.19} which requires an algebraically closed field. Moreover the Hilbert-Burch theorem in \Cref{proposition: n=3 hilbert burch} only works over $\bbC$. On the other hand, many computations apply more broadly without assumptions on the underlying field, for example, the use of properties of resultants in \cite{Jou91} in the proof of \Cref{lemma: computations with resultant}.

In our research we found it useful to work out some specific examples via a software computation. We used the package \cite{Sta18} for Macaulay2.

\section{Resultants and characteristic polynomials of tensors}
\label{resultant}\label{section: resultants}\label{section: characteristic polynomial}
In this section we recall the notion of resultant and the definition of characteristic polynomial of a tensor. The resultant is a tool to determine whether $n$ given homogeneous polynomials $F_1,\dots,F_n$ in $n$ variables share a non-trivial root. More precisely the resultant is a polynomial in the coefficients of $F_1,\dots,F_n$ which vanishes if and only if $F_1,\dots,F_n$ have a non-zero common root. For example, $n$ homogeneous linear forms have a non-zero common root in $\C^n$ if and only if they are linearly dependent: in this case the resultant is the determinant of the $n\times n$ matrix of the coefficients of $F_1,\dots,F_n$. Another example is determining if the algebraic variety defined by a homogeneous polynomial $F$ has a singular point. This happens if and only if the partial derivatives of $F$ share a non-zero root. In this case the resultant of the derivatives of $F$ is called the discriminant of $F$. We refer to \cite{GKZ}*{Chapter 13}, \cite{Jou91}*{Proposition 2.3} and \cite{CLO05}*{Chapter 3} for the following definitions and results. 

Let $A$ be a commutative
ring and let $n$ be a positive integer. If $\alpha=(\alpha_1,\dots,\alpha_n)\in\N^n$ is a multi-index, then we denote by $x^\alpha$ the monomial $x_1^{\alpha_1}\cdots x_n^{\alpha_n}$. Given positive integers  $d_1,\dots,d_n$ we consider the \emph{universal ring of coefficients}
$$\bbU_{d_1,\dots,d_n}=  \bbZ[u_{i,\alpha}\mid i\in\{1,\dots,n\}\mbox{ and }
\alpha_1+\dots+\alpha_n=d_i].$$
Given $P\in\bbU_{d_1,\dots,d_n}$ and $n$ homogeneous polynomials
$$F_i=\sum_{
\alpha_1+\dots+\alpha_n=d_i}c_{i,\alpha}x^{\alpha},\dots,F_n=\sum_{
\alpha_1+\dots+\alpha_n=d_i}c_{n,\alpha}x^{\alpha}\in A[x_1,\dots,x_n],
$$ we denote by $P(F_1,\dots,F_n)$ the element of $A$ obtained by replacing each variable $u_{i,\alpha}$ with the corresponding coefficient $c_{i,\alpha}$. In this paper we will make use of $A=\C$ or $A=\C[\lambda]$.

\begin{teo}\label{teo:resultant}
If $d_1,\dots,d_n$ are positive integers, then there exists a unique polynomial $\Res_{d_1,\dots,d_n}\in\bbU_{d_1,\dots,d_n}$ satisfying the following properties:
\begin{enumerate}[leftmargin=*]
\item if $F_1,\dots,F_n\in A[x_1,\dots,x_n]$ are homogeneous of degrees $d_1,\dots,d_n$, then the system 
    $$F_1(x_1,\dots,x_n)=\dots=F_n(x_1,\dots,x_n)=0$$
    has a non-zero solution over $A$ if and only if $\Res(F_1,\dots,F_n)=0$;
    \item $\Res_{d_1,\dots,d_n}(x_1^{d_1},\dots,x_n^{d_n})=1$;
    \item $\Res_{d_1,\dots,d_n}$ is irreducible
    as a polynomial in $A[u_{i,\alpha}]$.
\end{enumerate}
\end{teo}
When the degrees are clear from the context, we write $\Res$ instead of $\Res_{d_1,\dots,d_n}$.
\begin{definition}\label{def:resultant}
Let $A$ be a commutative ring. Given $F_1,\dots,F_n\in A[x_1,\dots,x_n]$ homogeneous polynomials, the \textit{resultant} of $F_1,\dots,F_n$ is $\Res(F_1,\dots,F_n)\in A$.
\end{definition}

The next lemma collects a few properties of the resultant from \cite{Jou91}*{Proposition 2.3, Proposition 5.9} and \cite{GKZ}*{Chapter 13, Proposition 1.3}.

\begin{lemma}\label{lemma: multilinearity of resultant}\label{lemma: pull out constants from the resultant}
Let $A$ be a commutative ring and let $F_1,\dots,F_n\in A[x_1,\dots,x_n]$ be homogeneous polynomials.
\begin{enumerate}[leftmargin=*]
\item $\Res(F_1,\dots,F_n)$ is homogeneous in the entries of $F_1,\dots,F_n$.  In particular $$\Res(F_1,\dots,F_{i-1},\lambda F_i,F_{i+1},\dots,F_n)=\lambda^{d_1\cdots d_{i-1}d_{i+1}\cdots  d_n}\Res(F_1,\dots,F_n).$$ 
\item If $H\in  A[x_1,\dots,x_n]$ is homogeneous, then $\res(F_1,\dots,F_{j-1},F_j H,F_{j+1}\dots,F_n)$ equals
$$\res(F_1,\dots,F_{j-1},F_j,F_{j+1}\dots,F_n)\cdot\res(F_1,\dots,F_{j-1}, H,F_{j+1}\dots,F_n).$$
\item Let $j\in\{1,\dots,n\}$ and let $\{H_i\mid i\neq j\}\subseteq A[x_1,\dots,x_n]$ be a set of $n-1$ homogeneous polynomials. If $\deg(H_i)=\deg(F_j)-\deg(F_i)
$ for every $i\neq j$, then 
$$\Res\left(F_1,\dots,F_{j-1},F_j+\sum_{i\neq j}H_iF_i,F_{j+1},\dots,F_n\right)=\Res(F_1,\dots,F_n).$$
\end{enumerate}
\end{lemma}

Resultants are useful to  define the main objects of this paper: the characteristic polynomial and the eigenschemes of a tensor. We also recall the notions of multiplicity from the introduction.  As pointed out both in \cite{multiplicities}*{Section 5.2} and in  \cite{GGTV}*{Section 1} the definitions of eigenvalue and eigenvector
depend only on the algebraic maps defined by $T$ and $\bft$, rather than on $T$ and $\bft$ themselves.
For this reason we may regard the tensors $T$ and $\bft$ as tuples of degree $d$ homogeneous polynomials. Therefore it makes sense to talk about their resultant.

\begin{defn}\label{definition: characteristic polynomial}
Let $T,\bft\in \C^n\ot((\C^n)^*)^{\ot d}
$ such that $\Res(\bft)\neq 0$. The \emph{$\bft$-characteristic polynomial} of $T$ is $\f_{\bft,T}(\lambda)=\res(T-\lambda\bft)\in\C[\lambda]$. If $\lambda_0\in\C$, then the \emph{$\bft$-algebraic multiplicity} of $\lambda_0$ with respect to $T$ is the multiplicity of $\lambda_0$ as a root of $\f_{\bft,T}$.
\end{defn}
By \Cref{teo:resultant} the number $\lambda_0$ is a root of $\f_{\bft,T}$ if and only if 
$\lambda_0$ is a $\bft$-eigenvalue of $T$. If we denote by
\[D(n,d)=nd^{n-1},
\]
then \cite{FO14}*{Theorem 15} states that $\deg(\f_{\bft,T})=D(n,d)$.

\begin{remark}\label{remark: an eigenvector has only one eigenvalue}
In \cite{FO14}*{Section 1}, $\bft$ is called \emph{non-singular} if $\Res(\bft)\neq 0$. Since we assume $\bft$ to be non-singular, every $\bft$-eigenvector of $T$ has a unique $\bft$-eigenvalue. Indeed, 
our hypothesis implies that $\bft(w^{\ot d})\neq 0$ for every $w\in\C^n\setminus\{0\}$, so if $T(w^{\ot d})=\lambda\bft(w^{\ot d})=\mu \bft(w^{\ot d})$, then $\lambda=\mu$.
\label{remark: complete intersection}We can state the property $\Res(\bft)\neq 0$ in a different way. If we  
represent $\bft$ as a tuple of homogeneous degree $d$ polynomials $(G_1,\dots,G_n)$ in $n$ variables, then $\Res(\bft)\neq 0$ means that the only common zero of $G_1,\dots,G_n$ is the trivial one. In algebraic terms $G_1,\dots,G_n$ are a regular sequence.
\end{remark}

\begin{remark}
As pointed out in \cite{GGTV}*{Remark 2.3}, in the case $d=1$ of matrices the notion of $\bft$-eigenvalue is independent on $\bft$: if $T$ and $\bft$ are matrices, then our hypothesis $\res(\bft)\neq 0$ implies that $\bft$
is invertible. In this case
\[T(w)=\lambda\bft(w)\Leftrightarrow (\bft^{-1}\cdot T)(w)=\lambda w,\]
hence $w$ is a $\bft$-eigenvector of $T$ with $\bft$-eigenvalue $\lambda$
if and only if $w$ is a $\id$-eigenvector of $\bft^{-1}\cdot T$ with $\id$-eigenvalue $\lambda$.
\end{remark}

\begin{definition}\label{definition: multiplicity eigenscheme} 
Let $T,\bft\in \C^n\ot((\C^n)^*)^{\ot d}
$ such that $\Res(\bft)\neq 0$ and let $\lambda_0\in\C$. If we write
\[T=(F_1,\dots,F_n)\mbox{ and }\bft=(G_1,\dots,G_n)
\]
as tuples of homogeneous polynomials, then the  \textit{$\bft$-eigenscheme} of $\lambda_0$ is the affine subscheme $E_{\bft,T}(\lambda_0)$ of $\C^n$ defined by the ideal $(F_1-\lambda_0G_1,\dots,F_n-\lambda_0G_n)$. The elements of $E_{\bft,T}(\lambda_0)$ are $\bft$-eigenvectors of $T$ with $\bft$-eigenvalues $\lambda_0$. The \emph{$\bft$-geometric multiplicity} of $\lambda_0$ is $\gm_{\bft,T}(\lambda_0)=\dim(E_{\bft,T}(\lambda_0))$. The \emph{$\bft$-span multiplicity} of $\lambda_0$ is \[\sm_{\bft,T}(\lambda_0)=\dim\bigl(\langle w\in W\mid w \mbox{ is a $\bft$-eigenvector of eigenvalue }\lambda_0\rangle\bigr).\]
\end{definition}

Since $E_{\bft,T}(\lambda_0)$ is defined by homogeneous polynomials, the eigenscheme is an \emph{affine cone}, that is, it is closed under scalar multiplication. This allows the geometrically inclined reader to think at $E_{\bft,T}(\lambda_0)$ as a subscheme of $\p^{n-1}$, rather than $\C^n$.

There are two more notions we want to recall in order to make sense of \Cref{conjecture: our version of am vs sm} and \Cref{conjecture: our version of am vs gm}. The \emph{degree} of a subscheme $X$ of $\C^n$ (or $\p^{n-1}$) is the number of intersection points with a general linear space of complementary dimension. If $X$ is defined by only one polynomial $F$, then $\deg(X)=\deg(F)$. If $X$ is an affine cone of dimension 1 in $\C^n$ (or equivalently a 0-dimensional subscheme of $\p^{n-1}$), then $\deg(X)$ is the number of lines (or the number of points), counted with multiplicity. Since the degree is a positive integer, \Cref{conjecture: our version of am vs gm} is stronger than \Cref{conjecture: algebraic vs geometric multiplicity}.

The second concept we need is a refinement of the span multiplicity. By \cite{Hart}*{Exercise II.5.10} closed subschemes of $\p^{n-1}$ are in bijection with saturated homogeneous ideals of $\C[x_1,\dots,x_n]$. \begin{definition} The \emph{scheme-theoretic linear span} of an affine cone in $\C^{n}$
is the vanishing locus of the largest vector subspace of $\C[x_1,\dots,x_n]_1$ contained in its saturated ideal. If $T,\bft\in \C^n\ot((\C^n)^*)^{\ot d}
$ such that $\Res(\bft)\neq 0$ and let $\lambda_0\in\C$, we denote by $\lm_{\bft,T}(\lambda_0)$ the dimension of the scheme-theoretic linear span of $E_{\bft,T}(\lambda_0)$.
\end{definition}

As an example, the point $\{(0,0)\}\in\C^2$ corresponds to the saturated ideal $(x,y)$. Its scheme-theoretic linear span is defined by $x=y=0$, so it is a point, according to our expectation. On the other hand, the scheme-theoretic linear span of the scheme corresponding to the saturated ideal $(x,y^2)$ is the line $x=0$, so it has dimension 1. 
While $(x,y)$ and $(x,y^2)$ define the same set $\{(0,0)\}$, they correspond to different schemes. We can think of $(x,y^2)$ as two infinitely near points. The corresponding scheme has degree 2. The dimension of this scheme-theoretic linear span is at least the dimension of the linear span of the underlying set, hence \Cref{conjecture: our version of am vs sm} is stronger than \Cref{conjecture: am vs sm}.

As the following example shows, $E_{\bft,T}(\lambda)$ may have irreducible components of different dimension. In this case $\dim(E_{\bft,T}(\lambda))$ 
is the dimension 
of a maximal dimensional component.

\begin{example}\label{example: the eigenscheme is not equidimensional}
Let $\{e_1,e_2,e_3\}$ be a basis of $\C^3$. 
Let
\begin{align*}
T&=e_3\ot e^*_1\otimes e^*_2+2 e_1\ot e^*_1\otimes  e^*_3\in \C^3\ot(\C^3)^*\ot(\C^3)^*\mbox{ and }\\
\bft&=e_1\ot e^*_1\ot e_1^*+e_2\ot e^*_2\ot e_2^*+e_3\ot e^*_3\ot e_3^*\in \C^3\ot(\C^3)^*\ot(\C^3)^*.
\end{align*}   
The algebraic map associated to $T$ is $(x_1,x_2,x_3)\mapsto (2x_1x_3,0,x_1x_2)$, so we can identify $T$ with the triple $T=(2x_1x_3,0,x_1x_2)$. In the same way we can identify $\bft$ with $(x_1^2,x_2^2,x_3^2)$. Hence $n=3$, $d=2$ and $D(n,d)=12$. Then $\lambda_0$ is a $\bft$-eigenvalue of $T$ if and only if there exists $(v_1,v_2,v_3)\neq (0,0,0)$ such that
\[
\begin{cases}
2v_1v_3=\lambda_0 v_1^2\\
0=\lambda_0 v_2^2\\
v_1v_2=\lambda_0 v_3^2.
\end{cases}
\]
The only $\bft$-eigenvalue of $T$ is 0, so $\f_{\bft,T}(\lambda)=\lambda^{12}$. The corresponding eigenscheme is defined by the ideal $(x_1x_2,x_1x_3)$ and it is the union of the plane $V(x_1)$ and the line $ V(x_2,x_3)$. So $\am_{\bft,T}(0)=12$ and $\gm_{\bft,T}(0)=\dim(E_{\bft,T}(0))=2$. The eigenscheme is not contained in any hyperplane and the saturation of its ideal does not contain any linear form, so both $\sm_{\bft,T}(0)$ and 
$\lm_{\bft,T}(0)$ equal 2. Since both irreducible components of $E_{\bft,T}(0)$ have degree 1, \Cref{conjecture: our version of am vs sm} and \Cref{conjecture: our version of am vs gm} are satisfied.
\end{example}

\section{Eigenschemes of codimension at most 1}\label{section: geometric multiplicities n-1}
It may happen that every vector is an eigenvector. In that case it is simple to show that all the multiplicities behave as predicted by the conjectures.
\begin{remark}\label{remark: if gm=n then everything is satisfied}
Let $\lambda_0\in\bbC$ and let $T,\bft\in \bbC^n\otimes ((\bbC^n)^*)^{\otimes d}$ such that $\res(\bft)\neq 0$. By definition  $\gm_{\bft,T}(\lambda_0)\in\{0,\dots,n\}$. If $\gm_{\bft,T}(\lambda_0)=n$, then $E_{\bft,T}(\lambda_0)=\C^n$, so $\lm_{\bft,T}(\lambda_0=n$, $E_{\bft,T}(\lambda_0)$ is irreducible of degree 1 and every vector is a $\bft$-eigenvector of $T$ with $\bft$-eigenvalue $\lambda_0$. By \Cref{remark: an eigenvector has only one eigenvalue} $\lambda_0$ is the only $\bft$-eigenvalue of $T$, thus $\f_{\bft,T}(\lambda)=(\lambda-\lambda_0)^{D(n,d)}$. In particular \begin{align*}
\am_{\bft,T}(\lambda_0)=D(n,d)=nd^{n-1}=\deg(E_{\bft,T}(\lambda_0))\gm_{\bft,T}(\lambda_0)d^{\gm_{\bft,T}(\lambda_0)-1}\ge n=\lm_{\bft,T}(\lambda_0)
,\end{align*}
therefore \Cref{conjecture: our version of am vs sm} and \Cref{conjecture: our version of am vs gm} are satisfied.
\end{remark}

In this section we focus on the next case, when the geometric multiplicity is $n-1$. We provide bounds on the algebraic multiplicity and we employ them to prove that \Cref{conjecture: our version of am vs sm} and \Cref{conjecture: our version of am vs gm} hold under the assumption that the eigenscheme has codimension 1. The ideal of such varieties is contained in a principal ideal, that is an ideal generated by only one polynomial. This allows us to make some explicit computations with the resultant. 
Let us  formalize this idea.

\begin{remark}\label{Division}\label{remark: eigenscheme codim 1 principal ideal}
Let $\bft\in
\C^n\ot((\C^n)^*)^{\ot d}$ such that $\Res(\bft)\neq 0$. Given $\lambda_0\in\bbC$ we want to describe all the tensors $T$ such that $\dim (E_{\bft,T}(\lambda_0))\geq n-1$. Set $R=\bbC[x_1,\dots,x_n]$ and write 
\[
T=(F_1,\dots,F_n)\mbox{ and } \bft=(G_1,\dots,G_n)
\]
as two $n$-ples of polynomials of degree $d$. Since $\dim (E_{\bft,T}(\lambda_0))\geq n-1$, by \cite{Hart}*{Proposition I.1.13} there exists $H\in R$ such that
\[
(H)\supseteq I(E_{\bft,T}(\lambda_0))= (F_1-\lambda_0G_1,\dots,F_n-\lambda_0G_n),
\]
hence $H$ divides each generator of $I(E_{\bft,T}(\lambda_0))$. If we call $b=\deg(H)$, then this implies that $b\le d$ and there exist $M_1,\dots,M_n\in R_{d-b}$ such that
\begin{equation}\label{equation: F and G special form principal}
F_1=\lambda_0G_1+HM_1,\dots, F_n=\lambda_0G_n+HM_n.  
\end{equation}
Conversely, if $F_1,\dots,F_n$ are as in \eqref{equation: F and G special form principal}, then $V(H)$ is contained in $E_{\bft,T}(\lambda_0)$ and so $\dim (E_{\bft,T}(\lambda_0))\geq n-1$. Thus the set of the tensors $T$ such that $\dim (E_{\bft,T}(\lambda_0))\geq n-1$ can be identified with
\[
\calS_{n,d}(\lambda_0)=\{(\lambda_0G_1+HM_1,\dots,\lambda_0G_n+HM_n)\in R_d^{\oplus n}\;|\; H\in R_b,\,M_i\in R_{d-b}\text{ for some $b\in\N$} \}.
\]
In other words $\calS_{n,d}(\lambda_0)$ is the image of the morphism
\[
\begin{array}{cccc}
\Psi_{n,d,\lambda_0}:&\displaystyle\bigsqcup_{b=1}^d\left(R_b \oplus R_{d-b}^{\oplus n}\right) & \to & R_d^{\oplus n}\\
     &(H,M_1,\dots,M_n)&\mapsto & (\lambda_0G_1+HM_1,\dots,\lambda_0G_n+HM_n).
\end{array}
\]
\end{remark}

\Cref{Division} shows that $\calS_{n,d}(\lambda_0)$ is the image of an algebraic map, so it is  constructible by \cite{Hart}*{Exercise II.3.19}. While it has several irreducible components, indexed by $b$, on each component it makes sense to speak of the general tensor among those such that $\dim (E_{\bft,T}(\lambda_0))\geq n-1$ and $\deg(E_{\bft,T}(\lambda_0))=b$. The generality conditions on $\calS_{n,d}(\lambda_0)$ can be translated into generality conditions on the domain of the map. This is helpful in our proofs, provided that we find a way to pass from a statement about a general tensor to a statement about any tensor.

\begin{lemma}\label{lemma: from general to all of them}
Let $\bft\in \C^n\ot ((\C^n)^*)^{\ot d}$ such that $\res(\bft)\neq 0$. Let  $\lambda_0\in\C$ and $b,q\in\N$. Set
\[Y=\{T\in\C^n\ot ((\C^n)^*)^{\ot d}\mid \dim(E_{\bft,T}(\lambda_0))=n-1\mbox{ and }\deg(E_{\bft,T}(\lambda_0))=b\},\]
and suppose that
\begin{equation}\label{eq: disuguaglianza}
\am_{\bft,T}(\lambda_0)\ge q
\end{equation}
for the general element $T$ of $Y$. Then \eqref{eq: disuguaglianza} holds for every $T\in Y$.
\begin{proof}
By \Cref{Division} the set $Y$
is a dense subset of the irreducible set $\Psi_{n,d,\lambda_0}(R_d\oplus R_{d-b}^{\oplus n})$. Thus for every $T\in\Psi_{n,d,\lambda_0}(R_d\oplus R_{d-b}^{\oplus n})$ there exists a line $\{T_s\mid s\in\C\}$ such that $T_s\in Y$ for the general $s\in\C$ and $T_0=T$. Hence there exists a one-dimensional family of univariate polynomials $\f_{\bft,T_s}$ converging to $\f_{\bft,T}$. By hypothesis $(\lambda-\lambda_0)^q\mid \varphi_{\bft,T_s}$ for every $s\neq 0$. The coefficients of $\varphi_{\bft, T_s}$ are polynomials in the entries of $T_s$ and when these entries are specialized to those of $T$ we get $\varphi_{\bft,T}$
. In particular $(\lambda-\lambda_0)^q\mid\varphi_{\bft,T}$ and this concludes the proof.
\end{proof}
\end{lemma}
The next lemma describes the resultant when some of the polynomials have the form \eqref{equation: F and G special form principal}.

\begin{lemma}\label{lemma: computations with resultant}
Let $n\ge 2$ and $\lambda_0\in\C$. Let $H,M_1,M_2,G_1,G_2,F_3,\dots,F_n\in\C[x_1,\dots,x_n]$ be homogeneous polynomials. If $\res(G_2, HM_2,F_3,\dots,F_n)\neq 0$, then
\[
\res\begin{pmatrix}HM_1+(\lambda_0-\lambda)G_1\\HM_2+(\lambda_0-\lambda)G_2\\F_3\\\vdots\\F_n\end{pmatrix}=\frac{(\lambda_0-\lambda)^{\deg(H)\deg(F_3)\cdots\deg(F_n)}}{\res(G_2,M_2,F_3,\dots,F_n)}\cdot\res\begin{pmatrix}M_1G_2-M_2G_1\\HM_2+(\lambda_0-\lambda)G_2\\F_3\\\vdots\\F_n\end{pmatrix}.
\]
\begin{proof}
Apply \Cref{lemma: multilinearity of resultant}(3) and multiply the first polynomial by $G_2$ to get
\begin{align*}
\res\begin{pmatrix}HM_1+(\lambda_0-\lambda)G_1\\HM_2+(\lambda_0-\lambda)G_2\\F_3\\\vdots\\F_n\end{pmatrix}=\frac{\res(HM_1G_2+(\lambda_0-\lambda)G_1G_2,HM_2+(\lambda_0-\lambda)G_2,G_3,\dots,F_n)
}{\res(G_2,HM_2+(\lambda_0-\lambda)G_2,F_3,\dots,F_n)}.
\end{align*}
In the numerator we apply \Cref{lemma: multilinearity of resultant}(2) to replace $HM_1G_2+(\lambda_0-\lambda)G_1G_2$ by the algebraic combination \[HM_1G_2+(\lambda_0-\lambda)G_1G_2-G_1(HM_2+(\lambda_0-\lambda)G_2).\] In a similar way, in the denominator we replace the second polynomial $HM_2+(\lambda_0-\lambda)G_2$ by the linear combination $HM_2+(\lambda_0-\lambda)G_2-(\lambda_0-\lambda)G_2$, obtaining
\begin{align*}
\frac{\res(H(M_1G_2-M_2G_1),HM_2+(\lambda_0-\lambda)G_2,F_3,\dots,F_n)}{\res(G_2,HM_2,F_3,\dots,F_n)}.
\end{align*}
By applying again \Cref{lemma: multilinearity of resultant}, our fraction becomes
\begin{align*}
&\frac{\res(H,HM_2+(\lambda_0-\lambda)G_2,F_3,\dots,F_n)\res(M_1G_2-M_2G_1,HM_2+(\lambda_0-\lambda)G_2,F_3,\dots,F_n)}{\res(G_2,H,F_3,\dots,F_n)\res(G_2,M_2,F_3,\dots,F_n)}\\
&=\frac{\res(H,(\lambda_0-\lambda)G_2,F_3,\dots,F_n)\res(M_1G_2-M_2G_1,HM_2+(\lambda_0-\lambda)G_2,F_3,\dots,F_n)}{\res(G_2,H,F_3,\dots,F_n)\res(G_2,M_2,F_3,\dots,F_n)}.
\end{align*}
If we call $a=\deg(H)\deg(F_3)\cdots\deg(F_n)$ and we apply \Cref{lemma: pull out constants from the resultant}, then the latter equals
\begin{align*}
(\lambda_0-\lambda)^a&\frac{\res(H,G_2,F_3,\dots,F_n)\res(M_1G_2-M_2G_1,HM_2+(\lambda_0-\lambda)G_2,F_3,\dots,F_n)}{\res(G_2,H,F_3,\dots,F_n)\res(G_2,M_2,F_3,\dots,F_n)}\\
&=(\lambda_0-\lambda)^a\frac{\res(M_1G_2-M_2G_1,HM_2+(\lambda_0-\lambda)G_2,F_3,\dots,F_n)}{\res(G_2,M_2,F_3,\dots,F_n)}.\qedhere
\end{align*}
\end{proof}
\end{lemma}
\Cref{lemma: computations with resultant} tells us what happens if two of the polynomials have the form \eqref{equation: F and G special form principal}. Now we extend it to show what happens if there are more than two. For this purpose  we recursively define a function that we will use to bound the algebraic multiplicity.

\begin{definition}\label{definition: F recursive}
Let $b$, $d$ and $n$ be positive integers such that $n\ge 2$. For every $t\in\{1,\dots,n\}$ we define a function $A_{d,b,n,t}:\N^{n-t}\to\N$ recursively by letting
\begin{align*}
A_{d,b,n,0}(k_1,\dots,k_n)=A_{d,b,n,1}(k_2,\dots,k_n)=0
\end{align*}
and setting $A_{d,b,n,t}(k_{t+1},\dots, k_n)$ equal to \[bd^{t-2}k_{t+1}\cdots k_n+A_{d,b,n,t-1}(2d-b,k_{t+1},\dots, k_n)-A_{d,b,n,t-2}(d-b,d, k_{t+1}, \dots, k_n).\]
In particular, for $t=n$ we define $A_{d,b,n,n}=bd^{n-2}+A_{d,b,n,n-1}(2d-b)-A_{d,b,n,n-2}(d-b,d)\in\N$.\end{definition}

In \Cref{appendix: long formula} we work out some properties of  $A_{d,b,n,t}$ and we compute an explicit formula for $A_{d,b,n,n}$. We are ready to generalize \Cref{lemma: computations with resultant}.
\begin{proposition}\label{proposition: induction step algebraic multiplicity}
Let $b$, $d$ and $n$ be positive integers such that $b\le d$ and $n\ge 2$. Fix $G_1,\dots,G_s\in\C[x_1,\dots,x_n]_d$ and  let $H
,M_1,\dots,M_s,P_{s+1},\dots,P_n\in\C[x_1,\dots,x_n]$ be general polynomials such that $\deg(H)=b$ and $\deg(M_i)=d-b$ for every $i\in\{1,\dots,s\}$. Then the multiplicity of 0 as a root of 
\begin{equation}\label{equation: resultant with s times mu}
    \res(HM_1+\mu G_1,
\dots,
HM_s+\mu G_s,
P_{s+1},
\dots,
P_n)
\end{equation}equals $A_{d,b,n,s}(\deg(P_{s+1}),\dots, \deg(P_{n}))$.\end{proposition}
\begin{proof}
For every $i\in\{s+1,\dots,n\}$, call $d_i=\deg(P_i)$. We argue by induction on $s$. If $s=0$, then $\res(P_1,\dots,P_n)$ does not depend on $\mu$, so the multiplicity of 0 is $0=A_{d,b,n,0}(d_{1},\dots,d_n)$. If $s=1$, then our generality assumption guarantees that
$\res(HM_1,P_2,\dots,
P_n)\neq 0$, hence 0 is not a root of $\res(HM_1+\mu G_1,P_2,\dots,
P_n)$, thus its multiplicity is $0=A_{d,b,n,1}(d_{2},\dots,d_n)$. Now assume that $s\ge 2$. Up to a sign, \Cref{lemma: computations with resultant} implies that
\eqref{equation: resultant with s times mu} equals
\[\frac{\mu^{bd^{s-2}d_{s+1}\cdots d_n}}{\res\begin{pmatrix}HM_1+\mu G_1\\\vdots\\HM_{s-2}+\mu G_{s-2}\\G_{s-1},M_{s-1}\\P_{s+1},\dots, P_n\end{pmatrix}}\cdot\res\begin{pmatrix}HM_1+\mu G_1\\\vdots\\HM_{s-1}+\mu G_{s-1}\\M_{s-1}G_s- M_sG_{s-1}\\P_{s+1},\dots,P_n\end{pmatrix}.\]
By induction hypothesis the multiplicity of 0 as a root of the resultant at the numerator is $A_{d,b,n,s-1}(\deg(M_{s-1}G_s- M_sG_{s-1}),d_{s+1},\dots,d_n)$, while  the multiplicity as a root of the denominator is $A_{d,b,n,s-2}(\deg(M_{s-1}),\deg(G_{s-1}),d_{s+1},\dots,d_n)$, hence the multiplicity of 0 as a root of \eqref{equation: resultant with s times mu} equals
\begin{align*}
b&d^{s-2}d_{s+1}\cdots d_n+A_{d,b,n,s-1}(2d-b,d_{s+1},\dots,d_n)-A_{d,b,n,s-2}(d-b,d,d_{s+1},\dots,d_n)\\&=A_{d,b,n,s}(d_{s+1},\dots,d_n).\qedhere
\end{align*}
\end{proof}

We have all the ingredients to prove the main result of this section, which provides a sharp bound on the algebraic multiplicity of an eigenvalue when its eigenscheme has codimension 1.
\begin{theorem}\label{corollary: algebraic multiplicity when gm=n-1}
Let $T,\bft\in \C^n\ot ((\C^n)^*)^{\ot d}$ such that $\Res(\bft)\neq 0$ and let $X\subseteq \C^n$ be a general algebraic hypersurface. Let $\lambda_0\in\C$ and assume that $X$ is the union of all components of $E_{\bft,T}(\lambda_0)$ of maximal dimension. Then $\deg(X)\le d$ and $$\am_{\bft,T}(\lambda_0)\ge A_{d,\deg(X),n,n}
.$$
If moreover $T$ is general among tensors such that $X$ is the union of all maximal dimension components of $E_{\bft,T}(\lambda_0)$, then $\am_{\bft,T}(\lambda_0)= A_{d,\deg(X),n,n}$.
\begin{proof}
Let $b=\deg(X)$. By \Cref{remark: eigenscheme codim 1 principal ideal} $b\le d$ and we can write 
\begin{align*}
\bft&=(G_1,\dots,G_n)\mbox{ and } T=(HM_1+\lambda_0G_1,\dots,HM_n+\lambda_0G_n)
\end{align*}
for some homogeneous polynomials $H\in\C[x_1,\dots,x_n]_b$ and $M_1,\dots,M_n\in\C[x_1,\dots,x_n]_{d-b}$. Thanks to \Cref{lemma: from general to all of them} it is enough to prove that $\am_{\bft,T}(\lambda_0)= A_{d,b,n,n}$ under the assumption that $T$ is general. Therefore we assume that $H,M_1,\dots,M_n$ are general and we apply \Cref{proposition: induction step algebraic multiplicity} with $s=n$. Hence $\am_{\bft,T}(\lambda_0)$, which by definition is the multiplicity of $\lambda_0$ as a root of
\[\res\begin{pNiceMatrix}
F_1-\lambda G_1\\\vdots\\F_n-\lambda G_n
\end{pNiceMatrix}=\res\begin{pNiceMatrix}
HM_1+(\lambda_0-\lambda) G_1\\\vdots\\HM_n+(\lambda_0-\lambda) G_n
\end{pNiceMatrix},\]
equals $A_{d,b,n,n}$.
\end{proof}
\end{theorem}

\begin{corollary}\label{corollary: conjectures true when gm=n-1}
Let  $T,\bft\in \bbC^n\otimes ({\bbC^n}^*)^{\otimes d}$ such that $\res(\bft)\neq 0$ and let $\lambda_0\in\bbC$.
\begin{enumerate}[leftmargin=*]
\item 
If $\gm_{\bft,T}(\lambda_0)= n-1$, then $T$ satisfies  \Cref{conjecture: algebraic vs geometric multiplicity}\eqref{item: conjecture am gm weak} and \Cref{conjecture: our version of am vs sm}.
\item 
If moreover every irreducible component of $E_{\bft,T}(\lambda_0)$ has the same dimension, then $T$ satisfies \Cref{conjecture: our version of am vs gm} as well.
\end{enumerate}

\begin{proof}
By \Cref{corollary: algebraic multiplicity when gm=n-1} and  \Cref{lemma: Abnn is enough to beat geometric mult}  $\am_{\bft,T}(\lambda_0)\ge A_{d,b,n,n} \geq (n-1)bd^{n-2}\ge (n-1)d^{n-2}$, so \Cref{conjecture: algebraic vs geometric multiplicity}\eqref{item: conjecture am gm weak} holds. Consider now the linear span conjecture and call $b=\deg(E_{\bft,T}(\lambda_0))$. If $n\ge 3$, then $$\am_{\bft,T}(\lambda_0)\ge A_{d,b,n,n} \geq (n-1)d \geq n\ge \lm_{\bft,T}(\lambda_0).
$$
If $n=2$, then either
$\sm_{\bft,T}(\lambda_0)=
1\le \am_{\bft,T}(\lambda_0)$ and there is nothing else to prove, or $\lm_{\bft,T}(\lambda_0)=2$, so $b\ge 2$ and then we can apply again  \Cref{corollary: algebraic multiplicity when gm=n-1} and  \Cref{lemma: Abnn is enough to beat geometric mult} to obtain $\am_{\bft,T}(\lambda_0)\ge A_{d,b,2,2} \geq b\ge =\lm_{\bft,T}(\lambda_0)
$.

Now assume that every irreducible component of the eigenscheme has dimension $n-1$ and let $E_1,\dots,E_s$ be the irreducible components of $E_{\bft,T}(\lambda_0)$. Then 
$b=\deg(E_1)+\dots+\deg(E_s)$ and so
\begin{align*}
\am_{\bft,T}(\lambda_0) &\geq A_{d,b,n,n} 
\ge (n-1)bd^{n-2}
=\sum_{i=1}^s\deg(E_i)(n-1)d^{n-2},
\end{align*}
where the first inequality is \Cref{corollary: algebraic multiplicity when gm=n-1} and the second inequality is \Cref{lemma: Abnn is enough to beat geometric mult}.
\end{proof}
\end{corollary}

\section{Eigenvalue multiplicities for binary and ternary tensors}\label{section: multiplicities n=3}\label{section: multiplicities n=2}
In this section we focus on small values of $n$. Our results in \Cref{section: geometric multiplicities n-1} are already enough to give the complete picture for $n=2$.

\begin{corollary}\label{cor:generality_n=2}\label{theorem: conjecture multiplicity holds for n=2}
Let $\lambda_0\in\bbC$ and let $T,\bft\in \bbC^2\otimes ((\bbC^2)^*)^{\otimes d}$such that $\res(\bft)\neq 0$.
\begin{enumerate}[leftmargin=*]
\item If $E_{\bft,T}(\lambda_0)=\C^2$, then $\am_{\bft,T}(\lambda_0)=D(2,d)=2d$.
\item If $E_{\bft,T}(\lambda_0)\subsetneq\C^2$, then $\am_{\bft,T}(\lambda_0)\ge\deg( E_{\bft,T}(\lambda_0))$.
\item If $E_{\bft,T}(\lambda_0)\subsetneq\C^2$ and $T$
is general in the component of $\calS_{2,d}(\lambda_0)$ corresponding to $b=\deg( E_{\bft,T}(\lambda_0))$
, then $\am_{\bft,T}(\lambda_0)=\deg( E_{\bft,T}(\lambda_0))$.
\end{enumerate}
In particular $T$ satisfies \Cref{conjecture: our version of am vs sm} and \Cref{conjecture: our version of am vs gm}.
\end{corollary}
\begin{proof}
The first statement follows from \Cref{remark: if gm=n then everything is satisfied}. 
Assume now that $E_{\bft,T}(\lambda_0)\subsetneq\C^2$. Since $E_{\bft,T}(\lambda_0)$ is an affine cone, each of its components is a line and has codimension 1. Call $b=\deg(E_{\bft,T}(\lambda_0))$ and apply \Cref{corollary: algebraic multiplicity when gm=n-1} with $X=E_{\bft,T}(\lambda_0)$ to obtain \begin{align*}
\am_{\bft,T}(\lambda_0)&\ge A_{d,b,2,2}=b+A_{d,b,2,1}(2d-b)-A_{d,b,2,0}(d-b,d)=b,
\end{align*}
with equality if $T$ is general. Since $E_{\bft,T}(\lambda_0)$ is a union of lines, its degree $b$ is the number of its irreducible components, counted with multiplicities, so \Cref{conjecture: our version of am vs gm} holds. As for \Cref{conjecture: our version of am vs sm},
it is enough to apply \Cref{corollary: conjectures true when gm=n-1}.
\end{proof}

The next example shows that the inequality in \Cref{theorem: conjecture multiplicity holds for n=2} may be strict. 
\begin{example}
Let $\lambda_0=1$, $\bft=(G_1,G_2)$ and $T=(HM_1+G_1,HM_2+G_2)$, where
\begin{center}
\begin{tabular}{llc}
$G_1=(x_1-x_2)(x_1-2x_2)(x_1-4x_2)$     &  $G_2=(x_1+5x_2)(x_1-5x_2)(x_1+6x_2)$&\\
 $M_1=(x_1-x_2)(x_1+2x_2)$    & $M_2=(x_1-x_2)(x_1+x_2)$& $H=x_1+3x_2$.
\end{tabular}
\end{center}
Then $\deg( E_{\bft,T}(1 ))=\deg(H)=1$ and $\res(G_2,M_2)\neq 0$. By \Cref{lemma: computations with resultant}
\begin{align*}
\res\begin{pmatrix}HM_1+(1 -\lambda)G_1\\HM_2+(1 -\lambda)G_2\end{pmatrix}&=\frac{(1 -\lambda)^{\deg(H)}}{\res(G_2,M_2)}\cdot\res\begin{pmatrix}M_1G_2-M_2G_1\\HM_2+(1 -\lambda)G_2\end{pmatrix}\\
&=\frac{1 -\lambda}{\res(G_2,M_2)}\cdot\res\begin{pmatrix}M_1G_2-M_2G_1\\HM_2+(1 -\lambda)G_2\end{pmatrix}.
\end{align*}
Observe that $1 $ is a root of $\res(M_1G_2-M_2G_1,HM_2+(1 -\lambda)G_2)$, because by construction $M_1G_2-M_2G_1$ and $HM_2$ share the factor $x_1-x_2$. Hence $\am_{\bft,T}(1)>\deg(E_{\bft,T}(1))$.
\end{example}

\Cref{cor:generality_n=2} deals with  $2\times 2\times \dots\times 2$ tensors, so we move to the case $n=3$. We provide bounds on the algebraic multiplicity that imply \Cref{conjecture: our version of am vs sm} and \Cref{conjecture: algebraic vs geometric multiplicity}\eqref{item: conjecture am gm weak}. While we cannot fully prove the stronger \Cref{conjecture: our version of am vs gm}, our bounds are strictly better than  the ones predicted by \Cref{conjecture: algebraic vs geometric multiplicity}\eqref{item: conjecture am gm weak}.

When $n=3$ the eigenscheme can have irreducible components of codimension 1 and 2. For codimension 1 we can rely on \Cref{corollary: algebraic multiplicity when gm=n-1}. If the codimension is 2, the ideal of the eigenscheme is no longer principal. However, the saturated ideals of schemes of codimension 2 have a very precise structure: as described by the Hilbert-Burch theorem - see \cite{eisenbudsyzygies}*{Proposition 3.1 and Theorem 3.2} or
\cite{CLO05}*{Chapter 6, Proposition 2.6} - they are determinantal, namely they admit a set of generators consisting of the minors of a matrix of polynomials. Unfortunately the ideal 
introduced in \Cref{definition: multiplicity eigenscheme} is not always saturated
. This prevents us from applying such result to every codimension 2 eigenscheme. In the following remark we describe a class of tensors such that we are able to perform our computations.
\begin{remark}\label{remark: determinantal}
Let $\bft\in
\C^3\ot((\C^3)^*)^{\ot d}$ such that $\Res(\bft)\neq 0$. Given $\lambda_0\in\bbC$ we want to describe all the tensors $T$ such that the ideal defining $E_{\bft,T}(\lambda_0)$ is determinantal. More precisely, if we set $R=\C[x_1,x_2,x_3]$ and write
\[
T=(F_1,F_2,F_3)\mbox{ and } \bft=(G_1,G_2,G_3)
\]
as two elements of $R_d^{\oplus 3}$, we want to characterize the triples $(F_1,F_2,F_3)$ such that the ideal
$$(F_1-\lambda_0 G_1, F_2-\lambda_0G_2, F_3-\lambda_0G_3)$$
is generated by the $2\times 2$ minors of a matrix
$$
M=\begin{pmatrix}
A_1 & A_2 & A_3\\
B_1 & B_2 & B_3
\end{pmatrix}$$
with $A_1,A_2,A_3\in R_a$ and  $B_1,B_2,B_3\in R_{d-a}$ for some $1\leq a\leq d$. We denote the set of such tensors as $\calT_\bft(\lambda_0)\subset R_d^{\oplus 3}$. Since 
$$(F_1-\lambda_0 G_1, F_2-\lambda_0G_2, F_3-\lambda_0G_3)=(A_1B_2-A_2B_1,A_1B_3-A_3B_1,A_2B_3-A_3B_2),$$
there exists $P\in \GL_3(\bbC)$ such that 
$$
\begin{pmatrix}
F_1\\
F_2\\
F_3
\end{pmatrix}=P\begin{pmatrix}
A_1B_2-A_2B_1\\
A_1B_3-A_3B_1\\
A_2B_3-A_3B_2
\end{pmatrix}+\lambda_0\begin{pmatrix}
G_1\\
G_2\\
G_3
\end{pmatrix}.$$
It follows that the algebraic map
\[
\begin{array}{cccc}
\Theta_{d,\lambda_0}:&\displaystyle\GL_3(\C)\times\bigsqcup_{a=1}^d\left(R_a^{\oplus 3} \oplus R_{d-a}^{\oplus 3}\right) & \to & R_d^{\oplus 3}\\
     &(P,(A_1,A_2,A_3,B_1,B_2,B_3))&\mapsto & P\begin{pmatrix}
A_1B_2-A_2B_1\\
A_1B_3-A_3B_1\\
A_2B_3-A_3B_2
\end{pmatrix}
\end{array}
\]
is surjective on $\calT_\bft(\lambda_0)$. 
\end{remark}

As we argued in \Cref{section: geometric multiplicities n-1}, \Cref{remark: determinantal} shows that $\calT_{\bft}(\lambda_0)$ is the image of an algebraic map, so it is  constructible by \cite{Hart}*{Exercise II.3.19}. While it has several irreducible components, indexed by $a$, on each component it makes sense to speak of the general tensor among those such that $(F_1-\lambda_0 G_1, F_2-\lambda_0G_2, F_3-\lambda_0G_3)$ is determinantal for some $A_1,A_2,A_3\in R_a$ and  $B_1,B_2,B_3\in R_{d-a}$. The generality conditions on $\calT_{\bft}(\lambda_0)$ can be translated into generality conditions on the domain of the map. 

\begin{proposition}\label{proposition: n=3 hilbert burch} Let $\bft\in \C^3\otimes ((\C^3)^*)^{\otimes d}$ such that $\res(\bft)\neq 0$ and let $\lambda_0\in\C$. If $T\in\calT_\bft(\lambda_0)$, then $\am_{\bft,T}(\lambda_0)\ge \frac{d^2}{2}+1$.
\begin{proof}
Let $F_1,F_2,F_3\in\C[x_1,x_2,x_3]_d$ be the homogeneous polynomials representing $T$ and let $G_1,G_2,G_3\in\C[x_1,x_2,x_3]_d$ representing $\bft$. By hypothesis the ideal
\[
(F_1-\lambda_0G_1,F_2-\lambda_0G_2,F_3-\lambda_0G_3)\]
is determinantal. 
More precisely there exist 
a positive integer $a$ and six polynomials $A_1,A_2,A_3\in \C[x_1,x_2,x_3]_a$ and  $B_1,B_2,B_3\in \C[x_1,x_2,x_3]_{d-a}$ such that $$(F_1-\lambda_0G_1,F_2-\lambda_0G_2,F_3-\lambda_0G_3)=(A_1B_2-A_2B_1,A_1B_3-A_3B_1,A_2B_3-A_3B_2).$$
Just like in \Cref{corollary: algebraic multiplicity when gm=n-1}, thanks to \Cref{remark: determinantal} it is enough to prove the statement under the assumption that $T$ is general in its component of $T\in\calT_\bft(\lambda_0)$. 
By construction there exists $P\in\GL_3(\C)$ 
sending
\begin{align*}
F_1-\lambda_0G_1&\mapsto A_1B_2-A_2B_1\\
F_2-\lambda_0G_2&\mapsto A_1B_3-A_3B_1\\
F_3-\lambda_0G_3&\mapsto A_2B_3-A_3B_2.\end{align*}
Therefore the  $\bft$-characteristic polynomial of $T$ is
\begin{align*}
\res\begin{pNiceMatrix}
F_1-\lambda G_1\\F_2-\lambda G_2\\F_3-\lambda G_3\end{pNiceMatrix}=\res\begin{pNiceMatrix}
(\lambda_0-\lambda)G_1+A_1B_2-A_2B_1\\
(\lambda_0-\lambda)G_2+A_1B_3-A_3B_1\\
(\lambda_0-\lambda)G_3+A_2B_3-A_3B_2\end{pNiceMatrix}.
\end{align*}
In order to lighten the notation we denote $\mu=\lambda_0-\lambda$. Notice that the determinantal generators satisfy the relation
\[A_3(A_1B_2-A_2B_1)-A_2(A_1B_3-A_3B_1)+A_1(A_2B_3-A_3B_2)=0.
\]
In order to make use of this relation we apply \Cref{lemma: multilinearity of resultant}(2), multiplying the first polynomial by $A_3$ so that the $\bft$-characteristic polynomial of $T$ equals
\begin{align*}
\frac{\res\begin{pNiceMatrix}
\mu G_1A_3+A_3(A_1B_2-A_2B_1)\\
\mu G_2+A_1B_3-A_3B_1\\
\mu G_3+A_2B_3-A_3B_2\end{pNiceMatrix}}{\res\begin{pNiceMatrix}
A_3\\
\mu G_2+A_1B_3-A_3B_1\\
\mu G_3+A_2B_3-A_3B_2\end{pNiceMatrix}}.
\end{align*}
In the numerator we apply \Cref{lemma: multilinearity of resultant}(3) to replace the first polynomial 
by \[\mu G_1A_3+A_3(A_1B_2-A_2B_1)-A_2(\mu G_2+A_1B_3-A_3B_1)+A_1(\mu G_3+A_2B_3-A_3B_2).\]
In the denominator we also apply \Cref{lemma: multilinearity of resultant}(3) to replace the second polynomial 
by $\mu G_2+A_1B_3-A_3B_1-B_1A_3$ and the third polynomial 
by $\mu G_3+A_2B_3-A_3B_2-B_2A_3$. Hence the $\bft$-characteristic polynomial of $T$ becomes
\begin{align*}
\frac{\res\begin{pNiceMatrix}
\mu (G_1A_3-G_2A_2+G_3A_1)\\
\mu G_2+A_1B_3-A_3B_1\\
\mu G_3+A_2B_3-A_3B_2\end{pNiceMatrix}}{\res\begin{pNiceMatrix}
A_3\\
\mu G_2+A_1B_3\\
\mu G_3+A_2B_3\end{pNiceMatrix}}=\frac{\mu^{d^2}\res\begin{pNiceMatrix}
G_1A_3-G_2A_2+G_3A_1\\
\mu G_2+A_1B_3-A_3B_1\\
\mu G_3+A_2B_3-A_3B_2\end{pNiceMatrix}}{\res\begin{pNiceMatrix}
A_3\\
\mu G_2+A_1B_3\\
\mu G_3+A_2B_3\end{pNiceMatrix}}
\end{align*}
by \Cref{lemma: pull out constants from the resultant}(1). By generality of $T$ we can apply \Cref{lemma: computations with resultant} to obtain
\begin{align*}
\frac{\mu^{d^2}\res\begin{pNiceMatrix}
G_1A_3-G_2A_2+G_3A_1\\
\mu G_2+A_1B_3-A_3B_1\\
\mu G_3+A_2B_3-A_3B_2\end{pNiceMatrix}\res(G_3,A_2,A_3)}{\mu^{\deg(B_3)\deg(A_3)}\res(
A_3,
A_1G_3-A_2G_2,\mu G_3+A_2B_3)}.
\end{align*}
Again by our generality assumption, $\mu=0$ is not a root of
\[\res\begin{pNiceMatrix}
G_1A_3-G_2A_2+G_3A_1\\
\mu G_2+A_1B_3-A_3B_1\\
\mu G_3+A_2B_3-A_3B_2\end{pNiceMatrix}\]
nor of $\res(A_3,
A_1G_3-A_2G_2,
\mu G_3+A_2B_3)$. Without loss of generality,  assume that $1\le \deg(A_3)\le\frac{d}{2}$, hence
\begin{align*}
\am_{\bft,T}(\lambda_0)&=d^2-\deg(B_3)\deg(A_3)=d^2-(d-\deg(A_3))\deg(A_3)\\
&= d^2-d\deg(A_3)+\deg(A_3)^2\ge \frac{d^2}{2}+1.\qedhere
\end{align*}
\end{proof}
\end{proposition}
We stress that if $n=3$,  $\gm_{\bft,T}(\lambda_0)=1$ and the ideal from \Cref{definition: multiplicity eigenscheme} is saturated, then $T\in\calT_\bft(\lambda_0)$ by the Hilbert-Burch theorem, hence \Cref{proposition: n=3 hilbert burch} implies \Cref{theorem: summary contribution multiplicities}(5). The next example shows that we cannot apply \Cref{proposition: n=3 hilbert burch} to non-saturated ideals.
\begin{example}
Let $n=d=2$ and consider $T=(y^2,z^2,yz)$ and $\bft=(x^2,y^2,z^2)$. The Macaulay2 code
\begin{lstlisting}[language=code]
loadPackage "Resultants";
R=QQ[l][x,y,z]; f_1=y^2; f_2=z^2; f_3=y*z;
N={f_1-l*x^2,f_2-l*y^2,f_3-l*z^2}; factor resultant(N) 
I=ideal(f_1-x^2,f_2-y^2,f_3-z^2); J=sub(I,QQ[x,y,z]); dim J \end{lstlisting}
shows that 1 is a $\bft$-eigenvalue of $T$ with $\am_{\bft,T}(1)=2$ and $\gm_{\bft,T}(1)=1$. 
\end{example}

\begin{remark}\label{remark: hilbert burch does not imply the strong conjecture but it's better thatn the weak conjecture} Let us compare our bounds from  \Cref{remark: if gm=n then everything is satisfied}, \Cref{corollary: algebraic multiplicity when gm=n-1} and \Cref{proposition: n=3 hilbert burch} with the values predicted by 
the conjectures in the case $n=3$.
\begin{enumerate}[leftmargin=*]
\item If $\gm_{\bft,T}(\lambda_0)=3$, then \Cref{remark: if gm=n then everything is satisfied} implies that \Cref{conjecture: our version of am vs sm} and \Cref{conjecture: our version of am vs gm} hold.
\item If $\gm_{\bft,T}(\lambda_0)=2$ and every irreducible component of the eigenscheme has dimension 2, then \Cref{corollary: conjectures true when gm=n-1}(2) again implies that both \Cref{conjecture: our version of am vs sm} and \Cref{conjecture: our version of am vs gm} hold. If the eigenscheme has both components of dimension 1 and components of dimension 2, then \Cref{conjecture: our version of am vs sm} and \Cref{conjecture: algebraic vs geometric multiplicity}\eqref{item: conjecture am gm weak} hold by \Cref{corollary: conjectures true when gm=n-1}(1), but our bound is not strong enough to prove \Cref{conjecture: our version of am vs gm}, because it does not account for the components of dimension 1.
\item If $\gm_{\bft,T}(\lambda_0)=1$, then $E_{\bft,T}(\lambda_0)$ is a union of lines through the origin. \Cref{conjecture: our version of am vs gm} predicts that the algebraic multiplicity is at least the number of lines, counted with multiplicity, or equivalently that  $\am_{\bft,T}(\lambda_0)\ge \deg(E_{\bft,T}(\lambda_0))$. Our bound does not depend on $\deg(E_{\bft,T}(\lambda_0))$, but only on $d$. In particular \Cref{proposition: n=3 hilbert burch} does not imply \Cref{conjecture: our version of am vs gm}. However, our bound considerably improves the weaker version. In fact, \Cref{conjecture: algebraic vs geometric multiplicity}\eqref{item: conjecture am gm weak} only predicts that $\am_{\bft,T}(\lambda_0)\ge 1$. About the linear span, our bound implies that $\am_{\bft,T}(\lambda_0)\ge \frac{d^2}{2}+1\ge 3=n$, hence \Cref{conjecture: our version of am vs sm} holds under these assumptions.
\end{enumerate}\end{remark}

\section{Eigenvalues of low-rank tensors}\label{section: eigenvalues of low rank tensors}
In this section we relate the rank of a tensor with some properties of its eigenvalues and eigenvectors. We focus on tensors $T\in \C^n \ot ((\C^n)^*)^{\ot d}$ having rank less or equal than $n$, for which the eigenvalue 0 plays an important role. In this process we provide more evidence for \Cref{conjecture: algebraic vs geometric multiplicity} and \Cref{conjecture: our version of am vs sm}. 
A similar analysis has been carried out in \cite{multiplicities}*{Section 4} with different perspective and techniques,  by  considering the marginal rank of the tensor, but only for symmetric tensors and only for a specific choice of $\bft$. Moreover \cite{multiplicities} does not treat the linear span multiplicity.

By \Cref{definition: multiplicity eigenscheme} the geometric multiplicity of 0 is independent on $\bft$. However, the algebraic multiplicity is not, as the following example shows.

\begin{example}
Let $n=d=2$. Take  
$T=(x^2,0)$ and consider
\begin{align*}
\bft_1=(x^2+xy+y^2,x^2+y^2)
\mbox{ and }\bft_2=(x^2+xy+y^2,x^2+xy-y^2).
\end{align*}
The Macaulay2 code
\begin{lstlisting}[language=code]
loadPackage "Resultants"; A=QQ[l]; R=A[x,y];
t1={x^2+x*y+y^2,x^2+x*y}; t2={x^2+x*y+y^2,x^2+x*y-y^2};
S1={l*(x^2+x*y+y^2)-x^2,l*(x^2+x*y)};
S2={l*(x^2+x*y+y^2)-x^2,l*(x^2+x*y-y^2)};
resultant t1, resultant t2, resultant S1, resultant S2
\end{lstlisting}
shows that 
$\bft_1$ and $\bft_2$ are non-singular,  $\f_{\bft_1,T}(\lambda)=\lambda^4-\lambda^3$ and  $\f_{\bft_2,T}(\lambda)=4\lambda^4-6\lambda^3+\lambda^2$. Hence $\am_{\bft_1,T}(0)\neq \am_{\bft_2,T}(0)$.
\end{example}
Even though the algebraic multiplicity of $0$ can vary with $\bft$, for low-rank tensors we give an intrinsic lower bound depending only on the rank of $T$.  
\begin{proposition}\label{proposition: low rank high zero multipl}
Let 
$T,\bft\in \C^n\ot ((\C^n)^*)^{\ot d}$ such that $\Res(\bft)\neq 0$ and let $r=\rk(T)$. If $r\in\{0,\dots,n\}$, then $\am_{\bft,T}(0)\ge (n-r)d^{n-1}$, and equality holds if $T$ is general among tensors of rank $r$.
\end{proposition}
\begin{proof}
Since $\rk(T)=r$, up to a linear change of coordinates we can write  
$T=(F_1,\dots,F_r,0,\dots,0)$.
Moreover we write $\bft=(G_1,\dots,G_n)$. 
Then the system defining the eigenvalues is
\[
\begin{cases}
F_1-\lambda G_1
=0\\
\vdots\\
F_r-\lambda G_r
=0\\
\lambda G_{r+1}
=0
\\
\vdots\\
\lambda G_n
=0.
\end{cases}\]
Apply \Cref{lemma: pull out constants from the resultant} $n-r$ times to conclude that $\f_{\bft,T}$ equals
\begin{align*}
&\res(F_1
-\lambda G_1
,\dots, F_r
-\lambda G_r
,\lambda G_{r+1}
,\dots,\lambda G_n
)\\
&=\lambda^{(n-r)d^{n-1}}\res(F_1-\lambda G_1,\dots, F_r-\lambda G_r, G_{r+1},\dots, G_n),
\end{align*}
so $\am_{\bft,T}(0)\ge (n-r)d^{n-1}$. The inequality is strict if and only if
$0$ is a root of  $$\res(F_1-\lambda G_1,\dots, F_r-\lambda G_r, G_{r+1},\dots, G_n),$$
that is, $\res(F_1,\dots, F_r, G_{r+1},\dots, G_n)=0$. This is a closed condition on the entries of the polynomials $F_1,\dots,F_r$, that is on the entries of $T$.
\end{proof}
\Cref{example: the eigenscheme is not equidimensional} shows that the inequality of \Cref{proposition: low rank high zero multipl} can be strict. Moreover a general tensor having rank at least $n$ does not have 0 as an eigenvalue. Indeed, we can represent such a tensor by a tuple of general homogeneous polynomials $(F_1,\dots, F_n)$. In this situation $E_{\bft,T}(0)=
\{0\}$, because the polynomials are general. For this reason, checking if 0 is an eigenvalue of a given tensor is a way to understand if it is low rank.
Using a similar argument to that of \Cref{proposition: low rank high zero multipl}, we give also a lower bound for the geometric multiplicity of $0$ and an upper bound on the geometric multiplicities of non-zero eigenvalues.
\begin{proposition}\label{proposition: geometric multiplicity of low rank tensors}\label{cor: spanmult0}
Let $T,\bft\in \C^n\ot ((\C^n)^*)^{\ot d}$ such that $\Res(\bft)\neq 0$ and $\rk(T)=r\in\{0,\dots,n\}$.
\begin{enumerate}[leftmargin=*]
\item 
$\gm_{\bft,T}(0)\ge  n-r$ and equality holds if $T$ is general among tensors of rank $r$.
\item If $\lambda_0$ is a non-zero $\bft$-eigenvalue of $T$, then  $\gm_{\bft,T}(\lambda_0)\le  r$.
\end{enumerate}
\begin{proof}
As in the proof of \Cref{proposition: low rank high zero multipl}, let  $T=(F_1,\dots,F_r,0,\dots,0)$,  $\bft=(G_1,\dots,G_n)$ and the eigenpairs are defined by
\[
\begin{cases}
F_1-\lambda G_1
=0\\
\vdots\\
F_r-\lambda G_r
=0\\
\lambda G_{r+1}
=0
\\
\vdots\\
\lambda G_n
=0.
\end{cases}\] 
Hence $E_{\bft,T}(0)$ is defined by the ideal $(F_1,\dots,F_r)$, thus $\dim E_{\bft,T}(0)\geq n-r$ and equality holds for $T$ general among the tensors of rank $r$.
If $\lambda_0$ is a non-zero $\bft$-eigenvalue of $T$, then we can divide the last $n-r$ equation by $\lambda_0$. In this way we notice that $E_{\bft,T}(\lambda_0)\subseteq V(G_{r+1},\dots,G_n)$. By \Cref{remark: complete intersection}, the ideal $(G_{r+1},\dots,G_n)$ is a complete intersection, so
\[
\gm_{\bft,T}(\lambda_0)=\dim (E_{\bft,T}(\lambda_0))\le \dim(V(G_{r+1},\dots,G_n))=r.\qedhere
\]
\end{proof}
\end{proposition}

The results we proved in this section allow to prove our conjectures for low-rank tensors. 
\begin{cor}\label{corollary: evidence for conjecture multiplicity of zero} Let $T,\bft\in \C^n\ot ((\C^n)^*)^{\ot d}$ such that $\Res(\bft)\neq 0$ and $\rk(T)=r\in\{1,\dots,n\}$.
\begin{enumerate}[leftmargin=*]
\item If $d(n-r)\ge n-1$, then $\am_{\bft,T}(0)\ge \gm_{\bft,T}(0)d^{\gm_{\bft,T}(0)-1}$.
    \item If $T$ is general among tensors of rank $r$, then $\am_{\bft,T}(0)\ge \gm_{\bft,T}(0)d^{\gm_{\bft,T}(0)-1}$.
\item 
$\am_{\bft,T}(0)\geq d^{n-1}
$.\end{enumerate}
In particular, under these assumptions the zero eigenvalue satisfies \Cref{conjecture: algebraic vs geometric multiplicity}\eqref{item: conjecture am gm weak} and \Cref{conjecture: our version of am vs sm}.
\begin{proof}
\autoref{proposition: low rank high zero multipl} implies that
\begin{align*}
\am_{\bft,T}(0) &\ge(n-r)d^{n-1}\ge(n-1)d^{n-2}\ge(n-1)d^{\gm_{\bft,T}(0)-1}\ge\gm_{\bft,T}(0)d^{\gm_{\bft,T}(0)-1}.
\end{align*}
For the second part, if $T$ is general among tensors of rank $r$, then by \autoref{proposition: geometric multiplicity of low rank tensors}
\begin{align*}
\am_{\bft,T}(0) &=(n-r)d^{n-1}\ge\gm_{\bft,T}(0)d^{n-1}\ge\gm_{\bft,T}(0)d^{n-r-1}=\gm_{\bft,T}(0)d^{\gm_{\bft,T}(0)d^{n-r}-1}.
\end{align*}
Now we consider the last statement. If $d=1$ there is nothing to prove because it follows by the theory of matrices, so we can suppose $d\geq 2$. \Cref{proposition: low rank high zero multipl} implies that
\[\am_{\bft,T}(0)\ge (n-r)d^{n-1}\geq d^{n-1}\geq 2^{n-1}\geq n\ge \lm_{\bft,T}(0). \qedhere
\]
\end{proof}\end{cor}

\appendix
\section{Some properties of \texorpdfstring{$A_{d,b,t,n}$}{Adbtn}}\label{appendix: long formula}

\begin{lemma}\label{lemma:function_A_b_n_n}
Let $d$, $b$ and $n$ 
be positive integers such that $n\ge 2$. If $t\in\{1,\dots,n\}$, then
$$A_{d,b,n,t}(k_{t+1},\dots,k_n)=k_{t+1}\cdots k_n A_{d,b,t,t}.$$
\end{lemma}
\begin{proof}
We argue by induction on $t$. If $t=1$, then by definition $A_{d,b,n,1}=A_{d,b,1,1}=0$, so the statement holds. If $t=2$, then $A_{d,b,2,2}=b$ and
\begin{align*}
A_{d,b,n,2}(k_{3},\dots,k_n)&=bk_{3}\cdots k_n +A_{d,b,n,1}(2d-b,k_{3},\dots,k_n)+A_{d,b,n,0}(d-b,d,k_{3},\dots,k_n)\\
&=bk_3\cdots k_{n}=k_3\cdots k_{n}A_{d,b,2,2}. 
\end{align*}
Now let $t\ge 3$. 
By inductive hypothesis $A_{d,b,n,t}(k_{t+1},\dots, k_n)$ equals
\begin{align*}
b&d^{t-2}k_{t+1}\cdots k_n+A_{d,b,n,t-1}(2d-b,k_{t+1},\dots, k_n)-A_{d,b,n,t-2}(d-b,d, k_{t+1}, \dots, k_n)\\
&=bd^{t-2}k_{t+1}\cdots k_n+(2d-b)k_{t+1}\cdots k_nA_{d,b,t-1,t-1}-(d-b)d k_{t+1}\cdots k_nA_{d,b,t-2,t-2}\\
&=k_{t+1}\cdots k_n\bigl(bd^{t-2}+(2d-b)A_{d,b,t-1,t-1}-(d-b)dA_{d,b,t-2,t-2}\bigr)\\
&=k_{t+1}\cdots k_n\bigl(bd^{t-2}+A_{d,b,t,t-1}(2d-b)-A_{d,b,t,t-2}(d-b,d)\bigr)
=k_{t+1}\cdots k_nA_{d,b,t,t}.\qedhere
\end{align*}
\end{proof}

\begin{lemma}
\label{lemma:boundingDelta}
Let $b$, $d$ and $n$ be positive integers such that 
$n\ge 2$. Then $$A_{d,b,n,n}-dA_{d,b,n-1,n-1}\geq bd^{n-2}. $$
\end{lemma}

\begin{proof}
We argue by induction on $n$. If $n=2$, then
\[A_{d,b,2,2}-dA_{d,b,1,1} = b-0 = bd^{2-2}\] by \Cref{definition: F recursive}. Now assume that $n\ge 3$ and that $A_{d,b,n-1,n-1}-dA_{d,b,n-2,n-2}\geq bd^{n-3}$
. By \Cref{lemma:function_A_b_n_n} and the induction hypothesis
\begin{align*}
A_{d,b,n,n}&-dA_{d,b,n-1,n-1}
=bd^{n-2}+(2d-b)A_{d,b,n-1,n-1}-(d-b)dA_{d,b,n-2,n-2}-dA_{d,b,n-1,n-1}\\
&=bd^{n-2}+(d-b)(A_{d,b,n-1,n-1}-dA_{d,b,n-2,n-2})
\geq bd^{n-2}+(d-b)bd^{n-3} \geq bd^{n-2}.\qedhere
\end{align*}
\end{proof}

\begin{lemma}
\label{lemma: Abnn is enough to beat geometric mult}
Let $b$, $d$ and $n$ be positive integers such that $n\ge 2$. Then $$A_{d,b,n,n} \geq (n-1)bd^{n-2}.$$ 
\begin{proof}
We argue by induction on $n$. If $n=2$, then $A_{d,b,2,2} =b \geq (2-1)bd^{2-2}$. If $n\ge 3$, 
then \Cref{lemma:boundingDelta} and the induction hypothesis imply that
\[
A_{d,b,n,n} \geq bd^{n-2}+dA_{d,b,n-1,n-1}\geq  bd^{n-2}+d(n-2)bd^{n-3}= (n-1)bd^{n-2}.\qedhere
\]
\end{proof}
\end{lemma}

\begin{theorem}
Let $b$, $d$ and $n$ be positive integers such that $n\ge 2$. Then
\[
A_{d,b,n,n}=\sum_{s=0}^{\lfloor \frac{n}{2}\rfloor-1}(-1)^sb(d-b )^s\sum_{j=0}^{n-2s-2}\binom{j+s}{s}d^{n-j-s-2}(2d-b)^{j}.
\]
\end{theorem}
\begin{proof}
We argue by induction on $n$. If $n=2$, then $A_{d,b,2,2}=b+A_{d,b,2,1}(2d-b)-A_{d,b,2,0}(d-b,d)=b$, so the statement holds. If $n\ge 3$, then 
\Cref{lemma:function_A_b_n_n}  implies that
\begin{align*}
A_{d,b,n,n}&=bd^{n-2}+A_{d,b,n,n-1}(2d-b)-A_{d,b,n,n-2}(d-b,d)\\
&=bd^{n-2}+(2d-b)A_{d,b,n-1,n-1}-(d-b)dA_{d,b,n-2,n-2}.
\end{align*}
We distinguish two cases
. If $n=2p$ is even, then by the inductive hypothesis 
\begin{align*}
A_{d,b,n,n}&=bd^{2p-2}+(2d-b)\sum_{s=0}^{p-2}(-1)^sb(d-b)^s\sum_{j=0}^{2p-2s-3}\binom{j+s}{s}d^{2p-j-s-3}(2d-b)^{j}\\
&\hphantom{{}={}}-(d-b)d\sum_{s=0}^{p-2}(-1)^sb(d-b)^{s}\sum_{j=0}^{2p-2s-4}\binom{j+s}{s}d^{2p-j-s-4}(2d-b)^{j}\\
&=bd^{2p-2}+\sum_{s=0}^{p-2}(-1)^sb(d-b)^s\sum_{j=0}^{2p-2s-3}\binom{j+s}{s}d^{2p-j-s-3}(2d-b)^{j+1}\\
&\hphantom{{}={}}-\sum_{s=0}^{p-2}(-1)^sb(d-b)^{s+1}\sum_{j=0}^{2p-2s-4}\binom{j+s}{s}d^{2p-j-s-3}(2d-b)^{j}.
\end{align*}
Setting $s'=s+1$ in the second sum we obtain
\begin{align*}
A_{d,b,n,n}
&=bd^{2p-2}+\sum_{s=0}^{p-2}(-1)^sb(d-b)^s\sum_{j=0}^{2p-2s-3}\binom{j+s}{s}d^{2p-j-s-3}(2d-b)^{j+1}\\
&\hphantom{{}={}}+\sum_{s'=1}^{p-1}(-1)^{s'}b(d-b)^{s'}\sum_{j=0}^{2p-2s'-2}\binom{j+s'-1}{s'-1}d^{2p-j-s'-2}(2d-b)^{j}.
\end{align*}
Taking out the term  $s=0$ from the first sum and the term $s'=p-1$ from the second sum we obtain
\begin{align*}
A_{d,b,n,n}&=bd^{2p-2}+b\sum_{j=0}^{2p-3}d^{2p-j-3}(2d-b)^{j+1}+(-1)^{p-1}b(d-b)^{p-1}d^{p-1}\\
&\hphantom{{}={}}+\sum_{s=1}^{p-2}(-1)^sb(d-b)^s\sum_{j=0}^{2p-2s-3}\binom{j+s}{s}d^{2p-j-s-3}(2d-b)^{j+1}\\
&\hphantom{{}={}}+\sum_{s'=1}^{p-2}(-1)^{s'}b(d-b)^{s'}\sum_{j=0}^{2p-2s'-2}\binom{j+s'-1}{s'-1}d^{2p-j-s'-2}(2d-b)^{j}.
\end{align*}
Taking out also the term $j=0$ in the last sum we obtain
\begin{align*}
A_{d,b,n,n}&=bd^{2p-2}+b\sum_{j=0}^{2p-3}d^{2p-j-3}(2d-b)^{j+1}+(-1)^{p-1}b(d-b)^{p-1}d^{p-1}\\
&\hphantom{{}={}}+\sum_{s=1}^{p-2}(-1)^sb(d-b)^s\sum_{j=0}^{2p-2s-3}\binom{j+s}{s}d^{2p-j-s-3}(2d-b)^{j+1}\\
&\hphantom{{}={}}+\sum_{s'=1}^{p-2}(-1)^{s'}b(d-b)^{s'}d^{2p-s'-2}\\
&\hphantom{{}={}}+\sum_{s'=1}^{p-2}(-1)^{s'}b(d-b)^{s'}\sum_{j=1}^{2p-2s'-2}\frac{(j+s'-1)!}{(s'-1)!j!}d^{2p-j-s'-2}(2d-b)^{j}\end{align*}
Now we put together the terms
\[bd^{2p-2}+(-1)^{p-1}b(d-b)^{p-1}d^{p-1}+\sum_{s'=1}^{p-2}(-1)^{s'}b(d-b)^{s'}d^{2p-s'-2}
\]
to obtain
\begin{align*}A_{d,b,n,n}&=b\sum_{j=0}^{2p-3}d^{2p-j-3}(2d-b)^{j+1}+\sum_{s'=0}^{p-1}(-1)^{s'}b(d-b)^{s'}d^{2p-s'-2}\\
&\hphantom{{}={}}+\sum_{s=1}^{p-2}(-1)^sb(d-b)^s\sum_{j=0}^{2p-2s-3}\binom{j+s}{s}d^{2p-j-s-3}(2d-b)^{j+1}\\
&\hphantom{{}={}}+\sum_{s'=1}^{p-2}(-1)^{s'}b(d-b)^{s'}\sum_{j=1}^{2p-2s'-2}\frac{s'}{j}\binom{j+s'-1}{s'}d^{2p-j-s'-2}(2d-b)^{j}.
\end{align*}
In the first sum we set $j'=j+1$ and in the last sum we set $l=j-1$. Then
\begin{align*}
A_{d,b,n,n}&=b\sum_{j'=1}^{2p-2}d^{2p-j'-2}(2d-b)^{j'}+\sum_{s'=0}^{p-1}(-1)^{s'}b(d-b)^{s'}d^{2p-s'-2}\\
&\hphantom{{}={}}+\sum_{s=1}^{p-2}(-1)^sb(d-b)^s\sum_{j=0}^{2p-2s-3}\binom{j+s}{s}d^{2p-j-s-3}(2d-b)^{j+1}\\
&\hphantom{{}={}}+\sum_{s'=1}^{p-2}(-1)^{s'}b(d-b)^{s'}\sum_{l=0}^{2p-2s'-3}\frac{s'}{l+1}\binom{s'+l}{s'}d^{2p-l-s'-3}(2d-b)^{l+1}\\
&=b\sum_{j'=1}^{2p-2}d^{2p-j'-2}(2d-b)^{j'}+\sum_{s'=0}^{p-1}(-1)^{s'}b(d-b)^{s'}d^{2p-s'-2}\\
&\hphantom{{}={}}+\sum_{s=1}^{p-2}(-1)^sb(d-b)^s\sum_{j=0}^{2p-2s-3}\frac{s+j+1}{j+1}\binom{j+s}{s}d^{2p-j-s-3}(2d-b)^{j+1}\\
&=b\sum_{j=1}^{2p-2}d^{2p-j-2}(2d-b)^{j}+\sum_{s'=0}^{p-1}(-1)^{s'}b(d-b)^{s'}d^{2p-s'-2}\\&\hphantom{{}={}}+\sum_{s=1}^{p-2}(-1)^sb(d-b)^s\sum_{j=0}^{2p-2s-3}\binom{j+s+1}{s}d^{2p-j-s-3}(2d-b)^{j+1}\\
&=b\sum_{j=1}^{2p-2}d^{2p-j-2}(2d-b)^{j}+\sum_{s'=0}^{p-1}(-1)^{s'}b(d-b)^{s'}d^{2p-s'-2}\\
&\hphantom{{}={}}+\sum_{s=1}^{p-2}(-1)^sb(d-b)^s\sum_{j'=1}^{2p-2s-2}\binom{j'+s}{s}d^{2p-j'-s-2}(2d-b)^{j'}.
\end{align*}

By including the first sum in the last one as the term corresponding to $s=0$ we conclude that $A_{d,b,n,n}$ equals
\begin{align*}
\sum_{s'=0}^{p-1}&(-1)^{s'}b(d-b)^{s'}d^{2p-s'-2}+\sum_{s=0}^{p-2}(-1)^sb(d-b)^s\sum_{j'=1}^{2p-2s-2}\binom{j'+s}{s}d^{2p-j'-s-2}(2d-b)^{j'}\\
&=\sum_{s=0}^{p-1}(-1)^{s}b(d-b)^{s}d^{2p-s-2}+\sum_{s=0}^{p-1}(-1)^sb(d-b)^s\sum_{j'=1}^{2p-2s-2}\binom{j'+s}{s}d^{2p-j'-s-2}(2d-b)^{j'},
\end{align*}
because in the last sum the term corresponding to $s=p-1$ equals $0$. Finally, including the terms in the first summation in the last one as the terms corresponding to $j=0$, we get the desired formula
\begin{align*}
A_{d,b,n,n}&=\sum_{s=0}^{p-1}(-1)^sb(d-b)^s\sum_{j=0}^{2p-2s-2}\binom{j+s}{s}d^{2p-j-s-2}(2d-b)^{j}\\
&=\sum_{s=0}^{\frac{n}{2}-1}(-1)^sb(d-b)^s\sum_{j=0}^{n-2s-2}\binom{j+s}{s}d^{n-j-s-2}(2d-b)^{j}.
\end{align*}
The case where $n$ is odd is analogous. 
\end{proof}
\begin{corollary}
Let $d,b$ and $n$ be positive integers such that $n\ge 2$. Then
\[
A_{d,b,n,t}(k_{t+1},\dots,k_n)=k_{t+1}\cdots k_n\biggl(\sum_{s=0}^{\lfloor \frac{t}{2}\rfloor-1}(-1)^sb(d-b )^s\sum_{j=0}^{t-2s-2}\binom{j+s}{s}d^{t-j-s-2}(2d-b)^{j}\biggr)
\]
for every $t\in\{1,\dots,n\}$.
\end{corollary}

\bibliographystyle{alpha}
\bibliography{tensor_eigenvalues_references.bib}
\end{document}